\documentclass[12pt]{article}
\usepackage[utf8]{inputenc}
\usepackage[T1]{fontenc}
\usepackage[english]{babel}
\usepackage{amsmath, amssymb, amsthm}            
\usepackage{amstext, amsfonts, a4}
\usepackage{graphicx}
\usepackage{hyperref}
\usepackage{color}
\usepackage{stmaryrd}
\usepackage{enumitem}
\usepackage{tikz}
\usepackage{mathrsfs}
\usepackage{pgfplots}
\pgfplotsset{compat=1.15}
\usetikzlibrary{arrows}
\usepackage{comment}

\theoremstyle{plain}
	\newtheorem{Theo}{Theorem}[subsection] 
	\newtheorem{Lem}[Theo]{Lemma}            

        \newtheorem{TheoPrinc}{Theorem}     

\theoremstyle{definition}

\theoremstyle{remark}
	\newtheorem{Rema}[Theo]{Remark}

\def\RR{{\mathbb R}}    
\newcommand{\Int}{\mbox{Int}}
\newcommand{\Vol}{\mbox{Vol}}
\newcommand{\KVol}{\mbox{KVol}}
\def\KK{{\mathcal{K}}}
\def\Lgn{{L_g^n}}
\def\Hgn{{H_g^n}}
\def\Mgn{{M_g^n}}

\title{Lower bound for KVol on the minimal stratum of translation surfaces}
\author{Julien Boulanger}

\begin{document}
 \maketitle

\begin{abstract}
In this paper we are interested in algebraic intersection of closed curves of a given length on translation surfaces. We study the quantity KVol, defined in \cite{CKM} and studied in \cite{CKM, CKMcras, BLM22, Bou23}, and we construct families of translation surfaces in each connected component of the minimal stratum $\mathcal{H}(2g-2)$ of the moduli space of translation surfaces of genus $g \geq 2$ such that KVol is arbitrarily close to the genus of the surface, which is conjectured to be the infimum of KVol on $\mathcal{H}(2g-2)$.
\end{abstract}

\section{Introduction}
Given a closed oriented surface $X$, the algebraic intersection $\Int(\cdot,\cdot)$ defines a symplectic bilinear form on the first homology group $H_1(X, \RR)$. When $X$ is endowed with a Riemannian metric, we can define the quantity
\[
\KVol(X) := \Vol(X) \sup_{\alpha,\beta} \frac{\Int(\alpha,\beta)}{l_g(\alpha) l_g(\beta)}
\]
where the supremum ranges over all piecewise smooth closed curves $\alpha$ and $\beta$ in $X$. Here, $\Vol(X)$ denotes the Riemannian volume, and $l_g(\alpha)$ (resp. $l_g(\beta)$) denotes the length of $\alpha$ (resp. $\beta$) with respect to the metric. \newline

The study of KVol originates in the work of Massart \cite{these_massart} and Massart-Muetzel \cite{MM}. In fact, KVol is also well defined if the Riemannian metric has isolated singularities, and it has been studied recently specifically in the case of translation surfaces (see \cite{CKM}, \cite{CKMcras}, \cite{BLM22}, \cite{Bou23}) for which one could hope to get explicit computations of KVol.

Although it is easy to make KVol go to infinity by pinching a non-separating curve, it cannot be made arbitrarily small: Massart and Muetzel \cite{MM} showed that for any closed oriented surface $X$ with a Riemannian metric, we have $\KVol(X) \geq 1$ with equality if and only if $X$ is a torus and the metric is flat. In light of this result, it is interesting to wonder what are the Riemannian (resp. hyperbolic, flat) surfaces of fixed genus $g$ having small KVol. This question turns out to be difficult to answer. In \cite{MM}, KVol is studied as a function over the moduli space of hyperbolic surfaces of fixed genus: they provide asymptotic bounds when the systolic length goes to zero. In \cite{CKMcras} Cheboui, Kessi and Massart extend the study of KVol to the moduli space of translation surfaces of genus $2$ having a single singularity. Namely, they investigate the following quantity:
\[ \KK(\mathcal{H}(2)) := \inf_{X \in \mathcal{H}(2)} \KVol(X) \]
In particular, they conjecture that $\KK(\mathcal{H}(2)) = 2$ and show that $\KK(\mathcal{H}(2)) \leq 2$ by exhibiting a family of (square-tiled) translation surfaces $L(n+1,n+1)$ having $\KVol$ converging to $2$ as $n$ goes to infinity.\newline

In this note, we tackle the same question in any genus $g \geq 2$. More precisely, we conjecture that:
\begin{equation*}
\KK(\mathcal{H}(2g-2)) := \inf_{X \in \mathcal{H}(2g-2)} \KVol(X) = g
\end{equation*}
and we construct surfaces in $\mathcal{H}(2g-2)$ having their KVol arbitrarily close to $g$, showing:
\begin{TheoPrinc}\label{theo:main}
For all $g \geq 2$, 
\[ \KK(\mathcal{H}(2g-2)) \leq g. \]
\end{TheoPrinc}

It has to be remarked that translation surfaces with a single singularity are very specific surfaces and that the infimum of KVol over all Riemannian surfaces of genus $g$ does not grow linearly with the genus as it is expected in the case of $\mathcal{H}(2g-2)$. In particular, as suggested by Sabourau to the author, a construction of \cite{Buser_Sarnak} gives a surface $X_g$ for each genus $g \geq 1$ such that 
\[ \KVol(X) \leq C \frac{g}{\log(g+1)^2}\]
for a given constant $C > 0$. This bound can be obtained using Theorem 1.5 of \cite{MM}, which compares KVol and the systolic volume, and the fact that the (homological) systolic volume of the surfaces constructed in \cite{Buser_Sarnak} grows as $C' \cfrac{g}{\log(g+1)^2}$.

However, in the case of translation surfaces having a single singularity, it is not possible to construct similar surfaces, as Boissy and Geninska \cite{Boissy_Geniska} (and independantly Judge and Parlier \cite{Judge_Parlier}) showed that in this setting the systolic volume has a linear bound in the genus. This is the reason why we expect the infimum of KVol over $\mathcal{H}(2g-2)$ to grow linearly with $g$.\newline

\begin{Rema}
Concerning the lower bound on KVol, Theorem 1.1 of \cite{BPS} gives directly that for any constant $A>0$, there exist $c_A > 0$ such that for any Riemannian surface $X$ of genus $g$ and such that $\mathrm{SysVol}(X) < A$, we have:
\[  c_A \frac{g}{\log(g+1)^2} \leq \KVol(X) .\]
It would be interesting to know whether the same inequality holds with a universal constant $c>0$ which does not depend on $A$. It should be noted that such a result has recently been shown for hyperbolic surfaces in the case where the algebraic intersection is replaced by the geometric intersection, see \cite{Torkaman}. The proof in this later case relies on the existence of a short figure eigth geodesic.
\end{Rema}

\paragraph{Connected components of $\mathcal{H}(2g-2)$.} With Theorem \ref{theo:main} in mind, it is interesting to wonder whether the bound $g$ can be achieved in any connected component of $\mathcal{H}(2g-2)$. Kontsevich and Zorich \cite{Kontsevich_Zorich} classified the connected components of any stratum of translation surfaces, and showed in particular that for any $g \geq 4$, $\mathcal{H}(2g-2)$ has three connected components: the hyperelliptic component $\mathcal{H}^{hyp}(2g-2)$, and two other connected components $\mathcal{H}^{even}(2g-2)$ and $\mathcal{H}^{odd}(2g-2)$ distinguished by the spin invariant. In genus $2$, the only connected component is hyperelliptic while in genus $3$ there are two connected components : odd spin and hyperelliptic. It turns out that the family of surfaces we construct in Section \ref{sec:construction_Lgn} belongs to odd spin for any $g \geq 2$. In Section \ref{sec:hyperelliptic_even} we give a family of hyperelliptic surfaces $\Hgn$ and even spin $\Mgn$ surfaces such that both $\KVol(\Hgn)$ and $\KVol(\Mgn)$ converge to $g$ as $n$ goes to infinity. In particular, we show:
\begin{TheoPrinc}
\begin{itemize}
\item $\KK(\mathcal{H}^{hyp}(2g-2)) \leq g$ for any $g \geq 2$.
\item $\KK(\mathcal{H}^{odd}(2g-2)) \leq g$ for any $g \geq 3$.
\item $\KK(\mathcal{H}^{even}(2g-2)) \leq g$ for any $g \geq 4$.
\end{itemize}
\end{TheoPrinc}

We assume familiarity with the geometry of translation surfaces, and encourage the reader to check out the surveys \cite{survey_Zorich}, \cite{survey_Wright} and \cite{survey_Massart}.

\paragraph{Acknoledgments.}
I would like to thank E.~Lanneau, D.~Massart and S.~Sabourau for useful and enlightening discussions related to the work presented here.

\section{Proof of Theorem \ref{theo:main}}\label{sec:construction_Lgn}
In this section, we prove Theorem \ref{theo:main} by exhitibing a family of surfaces $\Lgn$ for $g,n \geq 2$, (having odd spin parity and) such that $\Lgn$ has genus $g$ for each $n \geq 2$ and 
\[ \lim_{n \to \infty} \KVol(\Lgn) = g. \]

\subsection{Construction of the surface $\Lgn$.}
Given $g \geq 2$ and $n \geq 2$, define $\Lgn$ as the $(g(n+1)-1)$-square translation surface of genus $g$ with a single conical point which forms a staircase with steps of lengths and height $n$, as in Figure \ref{fig:examples_Sgn}.

\begin{figure}[h]
\center
\begin{tikzpicture}[line cap=round,line join=round,>=triangle 45,x=1cm,y=1cm, scale = 0.6]
\clip(-2,-4.1) rectangle (6,5.1);
\draw [line width=2pt] (-1,0)-- (4,0);
\draw [line width=1pt] (-1,1)-- (0,1);
\draw [line width=1pt] (-1,2)-- (0,2);
\draw [line width=1pt] (-1,3)-- (0,3);
\draw [line width=1pt] (-1,4)-- (0,4);
\draw [line width=2pt] (-1,5)-- (0,5);
\draw [line width=2pt] (-1,0)-- (-1,5);
\draw [line width=2pt] (0,1)-- (0,5);
\draw [line width=1pt] (0,0)-- (0,1);
\draw [line width=1pt] (1,0)-- (1,1);
\draw [line width=1pt] (2,0)-- (2,1);
\draw [line width=1pt] (3,0)-- (3,1);
\draw [line width=1pt] (4,0)-- (4,1);
\draw [line width=2pt] (5,-4)-- (5,1);
\draw [line width=2pt] (0,1)-- (5,1);
\draw [line width=1pt] (4,0)-- (5,0);
\draw [line width=1pt] (4,-1)-- (5,-1);
\draw [line width=1pt] (4,-2)-- (5,-2);
\draw [line width=1pt] (4,-3)-- (5,-3);
\draw [line width=2pt] (4,-4)-- (5,-4);
\draw [line width=2pt] (4,-4)-- (4,0);
\draw [line width=1pt, to-to] (-1.3,1)--(-1.3,5);
\draw (-1.2,3) node[left] {$n$};
\draw [line width=1pt, to-to] (0,-0.3) -- (4,-0.3);
\draw (2,-0.3) node[below] {$n$};
\draw [line width=1pt, to-to] (5.3,0) -- (5.3,-4);
\draw (5.3,-2) node[right] {$n$};
\draw (-1,-2.5) node[right] {$n=4$};
\draw (-1,-3.2) node[right] {$g=3$};
\end{tikzpicture}
\begin{tikzpicture}[line cap=round,line join=round,>=triangle 45,x=1cm,y=1cm, scale = 0.6]
\clip(-2,-5) rectangle (7.1,4.1);
\draw [line width=2pt] (-1,0)-- (3,0);
\draw [line width=1pt] (-1,1)-- (0,1);
\draw [line width=1pt] (-1,2)-- (0,2);
\draw [line width=1pt] (-1,3)-- (0,3);
\draw [line width=2pt] (-1,4)-- (0,4);
\draw [line width=2pt] (-1,0)-- (-1,4);
\draw [line width=2pt] (0,1)-- (0,4);
\draw [line width=1pt] (0,0)-- (0,1);
\draw [line width=1pt] (1,0)-- (1,1);
\draw [line width=1pt] (2,0)-- (2,1);
\draw [line width=1pt] (3,0)-- (3,1);
\draw [line width=2pt] (4,-3)-- (4,1);
\draw [line width=2pt] (3,-4)-- (3,0);
\draw [line width=2pt] (0,1)-- (4,1);
\draw [line width=1pt] (3,0)-- (4,0);
\draw [line width=1pt] (3,-1)-- (4,-1);
\draw [line width=1pt] (3,-2)-- (4,-2);
\draw [line width=1pt] (3,-3)-- (4,-3);
\draw [line width=2pt] (3,-4)-- (7,-4);
\draw [line width=1pt] (4,-3)-- (4,-4);
\draw [line width=1pt] (5,-3)-- (5,-4);
\draw [line width=1pt] (6,-3)-- (6,-4);
\draw [line width=2pt] (7,-3)-- (7,-4);
\draw [line width=2pt] (7,-3)-- (4,-3);
\draw [line width=1pt, to-to] (-1.3,1)--(-1.3,4);
\draw (-1.2,2.5) node[left] {$n$};
\draw [line width=1pt, to-to] (0,-0.3) -- (3,-0.3);
\draw (1.5,-0.3) node[below] {$n$};
\draw [line width=1pt, to-to] (4.3,0) -- (4.3,-3);
\draw (4.3,-1.5) node[right] {$n$};
\draw [line width=1pt, to-to] (4, -4.3) -- (7,-4.3);
\draw (5.5,-4.3) node[below] {$n$};
\draw (4, 3.2) node[right] {$n=3$};
\draw (4, 2.5) node[right] {$g=4$};
\end{tikzpicture}
\caption{The surface $L_3^4$ on the left, and $L_4^3$ on the right. The identifications are such that each horizontal (resp. vertical) rectangle is a cylinder.}
\label{fig:examples_Sgn}
\end{figure}
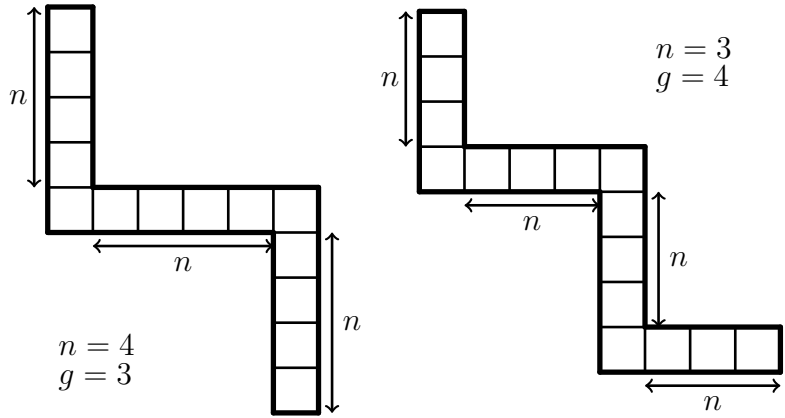

\paragraph{A basis of the homology.}
Let $e_1, \cdots , e_g$ (resp. $f_1, \cdots f_g$) be the horizontal (resp. vertical) saddle connections (see Figure \ref{fig:horiz_vert_sc}), seen as homology classes. Notice that for odd $i$, $e_i$ can be represented by a closed geodesic which do not pass through the singularity. We will refer to such homology classes as \emph{non singular} homology classes. This is also the case of $f_j$ for even $j$. On the contrary, for even $i$ (resp. odd $j$), the class $e_i$ (resp. $f_j$) will be called \emph{singular} as it can only be represented by closed geodesics passing through the singularity.

\begin{figure}
\center
\begin{tikzpicture}[line cap=round,line join=round,>=triangle 45,x=0.7cm,y=0.7cm]
\clip(-6.18,-6.98) rectangle (7,4.66);
\draw [line width=1pt] (-4,4)-- (-3,4);
\draw [line width=1pt] (-3,4)-- (-3,2);
\draw [line width=1pt] (-4,4)-- (-4,1.02);
\draw [line width=1pt] (-4,1.02)-- (-1,1);
\draw [line width=1pt] (-3,2)-- (0,2);
\draw [line width=1pt] (0,2)-- (0,0);
\draw [line width=1pt] (-1,1)-- (-1,-1);
\draw [line width=1pt,dash pattern=on 3pt off 3pt] (0,-1)-- (1,-2);
\draw [line width=1pt] (1,-3)-- (3,-3);
\draw [line width=1pt] (2,-2)-- (4,-2);
\draw [line width=1pt] (4,-2)-- (4,-5);
\draw [line width=1pt] (3,-3)-- (3,-6);
\draw [line width=1pt] (4,-5)-- (6,-5);
\draw [line width=1pt] (3,-6)-- (6,-6);
\draw [line width=1pt] (6,-6)-- (6,-5);
\draw (-3.5,4) node[above] {$e_1$};
\draw (-2,2) node[above] {$e_2$};
\draw (-4,1.5) node[left] {$f_2$};
\draw (-4,3) node[left] {$f_1$};
\draw (-0.5,2) node[above] {$e_3$};
\draw (-1,0) node[left] {$f_3$};
\draw (5,-6) node[below] {$e_g$};
\draw (3.5,-6) node[below] {$e_{g-1}$};
\draw (6,-5.5) node[right] {$f_g$};
\draw (4,-4) node[right] {$f_{g-1}$};
\draw (4,-2.5) node[right] {$f_{g-2}$};
\begin{scriptsize}
\draw [fill=black] (-4,4) circle (2pt);
\draw [fill=black] (-3,4) circle (2pt);
\draw [fill=black] (-3,2) circle (2pt);
\draw [fill=black] (-4,1.02) circle (2pt);
\draw [fill=black] (-1,1) circle (2pt);
\draw [fill=black] (0,2) circle (2pt);
\draw [fill=black] (-1,-1) circle (2pt);
\draw [fill=black] (3,-3) circle (2pt);
\draw [fill=black] (4,-2) circle (2pt);
\draw [fill=black] (4,-5) circle (2pt);
\draw [fill=black] (3,-6) circle (2pt);
\draw [fill=black] (6,-5) circle (2pt);
\draw [fill=black] (6,-6) circle (2pt);
\draw [fill=black] (-1,2) circle (2pt);
\draw [fill=black] (-3,1.0133333333333334) circle (2pt);
\draw [fill=black] (3,-2) circle (2pt);
\draw [fill=black] (4,-3) circle (2pt);
\draw [fill=black] (3,-5) circle (2pt);
\draw [fill=black] (4,-6) circle (2pt);
\draw [fill=black] (-4,2.0133333333333336) circle (2pt);
\end{scriptsize}
\end{tikzpicture}
\caption{The horizontal and vertical saddle connections $e_1, \cdots, e_g$, resp. $f_1, \cdots, f_g$.}
\label{fig:horiz_vert_sc}
\end{figure}
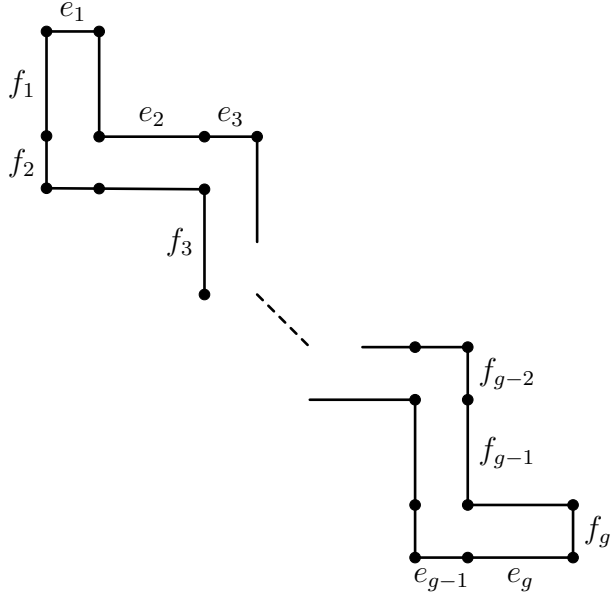

The intersection matrix of the $e_i$ and $f_j$ is given by the following table:

\begin{center}
\begin{tabular}{| c || c c c c c c c|}
\hline
 $\Int(e_i,f_j)$ & $e_1$ & $e_2$ & $e_3$ & $e_4$ & $e_5$ & $\cdots$ & $e_g$ \\ 
\hline\hline
 $f_1$ & 1 & -1 & 0  & 0 & 0 & $\cdots$ & 0 \\  
 $f_2$ & 0 & 1  & 0  & 0 & 0 & & 0 \\  
 $f_3$ & 0  & -1 & 1 & -1 & 0 & & 0 \\  
 $f_4$ & 0  & 0  & 0  & 1 & 0 & & 0 \\  
 $f_5$ & 0 & 0  & 0  & -1 & 1 & $\cdots$  & 0\\  
 $\cdots$ & &   &     &    &    & $\cdots$ & $\cdots$ \\  
 $f_g$ & 0  & 0 & 0 & 0 & 0 & $\cdots$ & 1 \\  
\hline
\end{tabular}
\end{center}

To see this, notice that for odd $i$, the fact that $e_i$ can be represented by a non-singular closed curve gives $\Int(e_i,f_j) = \delta_{i,j}$. The same holds for $f_j$ for even $j$. Next, given $i$ even, the holomogy class $e_{i-1} + e_i + e_{i+1}$ corresponds to a non-singular curve in $\Lgn$ which intersects $f_j$ if and only if $j=i$. In particular, 
\[ \Int(e_{i-1},f_j) + \Int(e_{i},f_j) + \Int(e_{i+1},f_j) = \delta_{i,j} \]
But the fact that both $i-1$ and $i+1$ are odd gives $\Int(e_{i-1},f_j) = \delta_{i-1,j}$ and $\Int(e_{i+1},f_j) = \delta_{i+1,j}$, so that
\[ \Int(e_{i},f_j) = \delta_{i,j} - \delta_{i-1,j} - \delta_{i+1,j} \]

Further, as $\Int(e_{i-1}+e_i+e_{i+1}, e_i) = 0$ for even $i$ (resp. $\Int(f_{j-1}+f_j+f_{j+1}, f_j) = 0$ for odd $j$), the same arguments gives that the $e_i$'s (resp. the $f_j$'s) do not intersect each other.\newline

As a concluding remark, notice that closed geodesics representing $e_1$ and $f_1$ are intersecting once and have respective length $1$ and $n$, and in particular:
\[ \KVol(\Lgn) \geq \Vol(\Lgn) \cdot \frac{\Int(e_1,f_1)}{l(e_1) l (f_1)} = (g(n+1)-1) \cdot \frac{1}{n}. \]

\paragraph{Computation of the spin.} As explained in \cite[Section 3]{Kontsevich_Zorich}, it is easy to compute the spin parity of an abelian differential $\omega$ given a symplectic basis of the first homology group $(a_i, b_i)_{1 \leq i \leq g}$ represented by smooth curves, and we have:
\[ \varphi (\omega) = \sum_{i=1}^g \Omega(a_i) \Omega(b_i) \mod 2. \]
where $\Omega (a_i) = ind_{a_i} +1$ and $2\pi \cdot ind_{a_i}$ is the total change of angle between the tangent vector to the curves and the horizontal foliation. Further, for any $a,b \in H^1(X,\omega)$, we have: 
\begin{equation}\label{eq:omega}
\Omega(a+b) = \Omega(a) + \Omega(b) + \Int(a,b).
\end{equation}
In the case of $\Lgn$, we use the basis $(a_i, b_i)_{1 \leq i \leq g}$ defined by:
\begin{align*}
a_i & = \left\{ 
\begin{array}{cc}
 e_i & \text{ if $i$ is even}  \\
e_{i-1} + e_i + e_{i+1} & \text{ if $i$ is odd}
\end{array} \right.  & \text{ and } & & b_i = f_i.
\end{align*}
The index of each $a_i$ is $0$ as well as the index of each $b_i$ for even $i$ because they correspond to non-singular homology classes. Further, using \eqref{eq:omega} we show that $\Omega(b_1) = 0$, as well as $\Omega(b_g) = 0$ if $g$ is odd, while $\Omega(b_i)=1$ for odd $i$, $1<i<g$. In particular, we deduce that the spin structure of $\Lgn$ has odd parity.\newline

Further, it should be remarked that although $L_2^n$ is hyperelliptic for any $n \geq 2$, $\Lgn$ is not hyperelliptic if $g \geq 3$. This is because an hyperelliptic involution would have to fix each cylinder, and hence must act as an involution of $R_1 \cup C_1$ (with the notations of Figure \ref{fig:another_model}) so that it must act as an involution on $C_1$, but it must also act as an involution of $C_1 \cup R_2 \cup C_2$, which is then impossible.

\paragraph{A useful model for the surfaces $\Lgn$.}
Let us finish this section by giving another model for $\Lgn$, which, although less intuitive at first sight, turns out to be helpful for the study of the intersections of saddle connections on $\Lgn$. This model is obtained from a cut and paste procedure which is described in Figure \ref{fig:another_model} in the example of $L_4^3$. The main idea is to glue together all the squares at the corners of $\Lgn$ to form a \emph{core staircase} to which are attached the \emph{long rectangles}. A general picture is given in Figure \ref{fig:another_model_2}.

\begin{figure}[h]
\center
\definecolor{zzttff}{rgb}{0.6,0.2,1}
\definecolor{ccqqqq}{rgb}{0.8,0,0}
\definecolor{ffvvqq}{rgb}{1,0.3333333333333333,0}
\definecolor{fuqqzz}{rgb}{0.9568627450980393,0,0.6}
\definecolor{qqqqff}{rgb}{0,0,1}
\definecolor{qqwuqq}{rgb}{0,0.39215686274509803,0}
\begin{tikzpicture}[line cap=round,line join=round,>=triangle 45,x=1cm,y=1cm]
\clip(-5.2,-4.2) rectangle (10.2,4.2);
\fill[line width=2pt,color=qqwuqq,fill=qqwuqq,fill opacity=0.2] (-5,1) -- (-4,1) -- (-4,4) -- (-5,4) -- cycle;
\fill[line width=2pt,color=qqqqff,fill=qqqqff,fill opacity=0.2] (-4,1) -- (-4,0) -- (-1,0) -- (-1,1) -- cycle;
\fill[line width=2pt,color=fuqqzz,fill=fuqqzz,fill opacity=0.25] (-1,0) -- (0,0) -- (0,-3) -- (-1,-3) -- cycle;
\fill[line width=2pt,color=ffvvqq,fill=ffvvqq,fill opacity=0.25] (0,-3) -- (0,-4) -- (3,-4) -- (3,-3) -- cycle;
\fill[line width=2pt,color=ccqqqq,fill=ccqqqq,fill opacity=0.25] (0,-4) -- (-1,-4) -- (-1,-3) -- (0,-3) -- cycle;
\fill[line width=2pt,color=zzttff,fill=zzttff,fill opacity=0.25] (0,0) -- (-1,0) -- (-1,1) -- (0,1) -- cycle;
\fill[line width=2pt,color=fuqqzz,fill=fuqqzz,fill opacity=0.25] (3.6,-0.6) -- (4.6,-0.6) -- (4.6,-3.6) -- (3.6,-3.6) -- cycle;
\fill[line width=2pt,color=qqwuqq,fill=qqwuqq,fill opacity=0.25] (5,-0.6) -- (5,-3.6) -- (6,-3.6) -- (6,-0.6) -- cycle;
\fill[line width=2pt,color=qqqqff,fill=qqqqff,fill opacity=0.2] (6.6,0) -- (6.6,1) -- (9.6,1) -- (9.598828611049326,0) -- cycle;
\fill[line width=2pt,color=ffvvqq,fill=ffvvqq,fill opacity=0.25] (5.6,1.4) -- (8.6,1.4) -- (8.6,2.4) -- (5.6,2.4) -- cycle;
\fill[line width=2pt,color=zzttff,fill=zzttff,fill opacity=0.1] (5,1) -- (5,0) -- (4,0) -- (4,1) -- cycle;
\fill[line width=2pt,color=ccqqqq,fill=ccqqqq,fill opacity=0.1] (5,1) -- (5,2) -- (4,2) -- (4,1) -- cycle;
\draw [line width=1pt] (-5,0)-- (-5,1);
\draw [line width=1pt] (-5,1)-- (-5,4);
\draw [line width=1pt] (-5,4)-- (-4,4);
\draw [line width=1pt] (-4,4)-- (-4,1);
\draw [line width=1pt,dash pattern=on 3pt off 3pt] (-4,1)-- (-4,0);
\draw [line width=1pt] (-4,0)-- (-5,0);
\draw [line width=1pt,dash pattern=on 3pt off 3pt] (-5,1)-- (-4,1);
\draw [line width=1pt] (-4,1)-- (-1,1);
\draw [line width=1pt] (-1,1)-- (0,1);
\draw [line width=1pt] (0,1)-- (0,0);
\draw [line width=1pt,dash pattern=on 3pt off 3pt] (0,0)-- (-1,0);
\draw [line width=1pt,dash pattern=on 3pt off 3pt] (-1,0)-- (-1,1);
\draw [line width=1pt] (-1,0)-- (-4,0);
\draw [line width=1pt] (-1,0)-- (-1,-3);
\draw [line width=1pt] (0,0)-- (0,-3);
\draw [line width=1pt,dash pattern=on 3pt off 3pt] (0,-3)-- (-1,-3);
\draw [line width=1pt] (-1,-3)-- (-1,-4);
\draw [line width=1pt] (-1,-4)-- (0,-4);
\draw [line width=1pt,dash pattern=on 3pt off 3pt] (0,-4)-- (0,-3);
\draw [line width=1pt] (0,-3)-- (3,-3);
\draw [line width=1pt] (3,-3)-- (3,-4);
\draw [line width=1pt] (3,-4)-- (0,-4);
\draw [line width=1pt] (3.6,-0.6)-- (4.6,-0.6);
\draw [line width=1pt] (4.6,-0.6)-- (4.6,-3.6);
\draw [line width=1pt,color=fuqqzz] (4.6,-3.6)-- (3.6,-3.6);
\draw [line width=1pt] (3.6,-3.6)-- (3.6,-0.6);
\draw [line width=1pt] (5,-0.6)-- (5,-3.6);
\draw [line width=1pt,color=qqwuqq] (5,-3.6)-- (6,-3.6);
\draw [line width=1pt] (6,-3.6)-- (6,-0.6);
\draw [line width=1pt] (6,-0.6)-- (5,-0.6);
\draw [line width=1pt] (6.6,0)-- (6.6,1);
\draw [line width=1pt] (6.6,1)-- (9.6,1);
\draw [line width=1pt,color=qqqqff] (9.6,1)-- (9.598828611049326,0);
\draw [line width=1pt] (9.598828611049326,0)-- (6.6,0);
\draw [line width=1pt] (5.6,1.4)-- (8.6,1.4);
\draw [line width=1pt,color=ffvvqq] (8.6,1.4)-- (8.6,2.4);
\draw [line width=1pt] (8.6,2.4)-- (5.6,2.4);
\draw [line width=1pt] (5.6,2.4)-- (5.6,1.4);
\draw [line width=1pt,dash pattern=on 3pt off 3pt] (5,1)-- (5,0);
\draw [line width=1pt] (5,0)-- (4,0);
\draw [line width=1pt,color=qqqqff] (4,0)-- (4,1);
\draw [line width=1pt,dash pattern=on 3pt off 3pt] (4,1)-- (5,1);
\draw [line width=1pt] (5,1)-- (5,2);
\draw [line width=1pt,color=fuqqzz] (5,2)-- (4,2);
\draw [line width=1pt,color=ffvvqq] (4,2)-- (4,1);
\draw [line width=1pt, to-to] (5,1.5)-- (5.6,1.9050508599237597);
\draw [line width=1pt,color=qqwuqq] (5,1)-- (6,1);
\draw [line width=1pt] (5,0)-- (6,0);
\draw [line width=1pt] (6,0)-- (6,1);
\draw [line width=1pt, to-to] (6,0.5)-- (6.6,0.5);
\draw [line width=1pt, to-to] (5.5,0)-- (5.5,-0.6);
\draw [line width=1pt, to-to] (4.5,0)-- (4.1,-0.6);
\draw (4,1.5) node[left] {$F_4$};
\draw (4,0.5) node[left] {$F_2$};
\draw (4.5,2) node[above] {$E_3$};
\draw (5.4,0.9) node[above] {$E_1$};
\draw (8.6,1.9) node[right] {$F_4$};
\draw (9.6,0.5) node[right] {$F_2$};
\draw (4.1,-3.6) node[below] {$E_3$};
\draw (5.5,-3.6) node[below] {$E_1$};
\draw (-4.9,2.5) node[right] {$R_1$};
\draw (-2.4,0.2) node[above] {$R_2$};
\draw (-0.9,-1.6) node[right] {$R_3$};
\draw (1.5,-3.2) node[below] {$R_4$};
\draw (5.5,-2) node[below] {$R_1$};
\draw (7.6,0.5) node[right] {$R_2$};
\draw (4.1,-2) node[below] {$R_3$};
\draw (6.6,1.9) node[right] {$R_4$};
\draw (-4.5,0.1) node[above] {$C_1$};
\draw (-0.5,0.1) node[above] {$C_2$};
\draw (-0.5,-3.2) node[below] {$C_3$};
\draw (5.5,0.1) node[above] {$C_1$};
\draw (4.5,0.1) node[above] {$C_2$};
\draw (4.5,1.1) node[above] {$C_3$};
\end{tikzpicture}
\caption{ $L_4^3$ and its alternative model.}
\label{fig:another_model}
\end{figure}
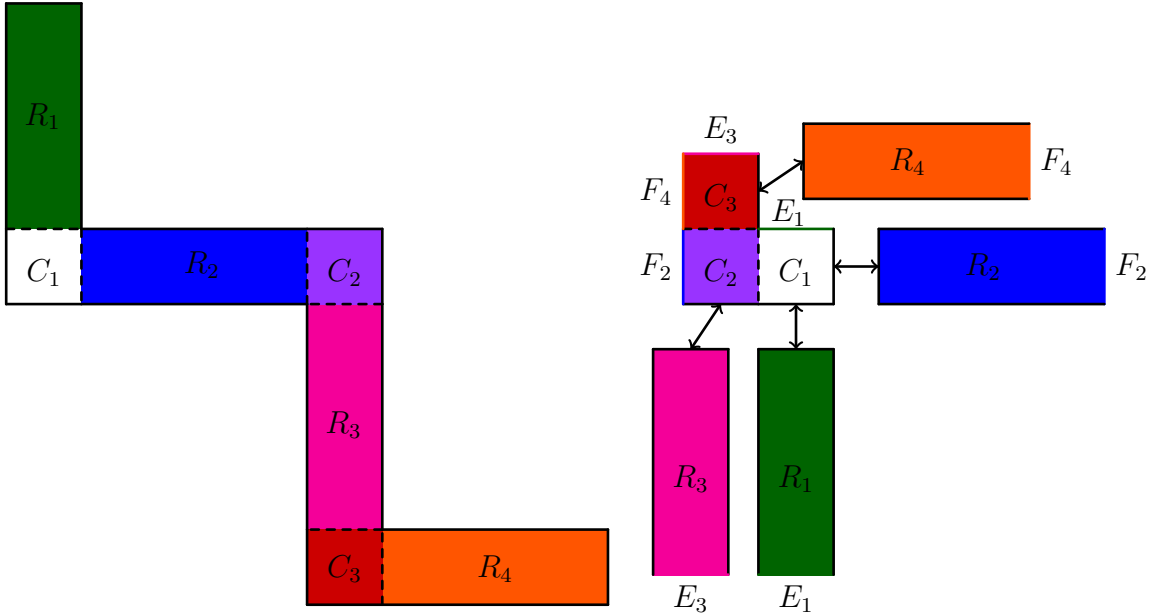

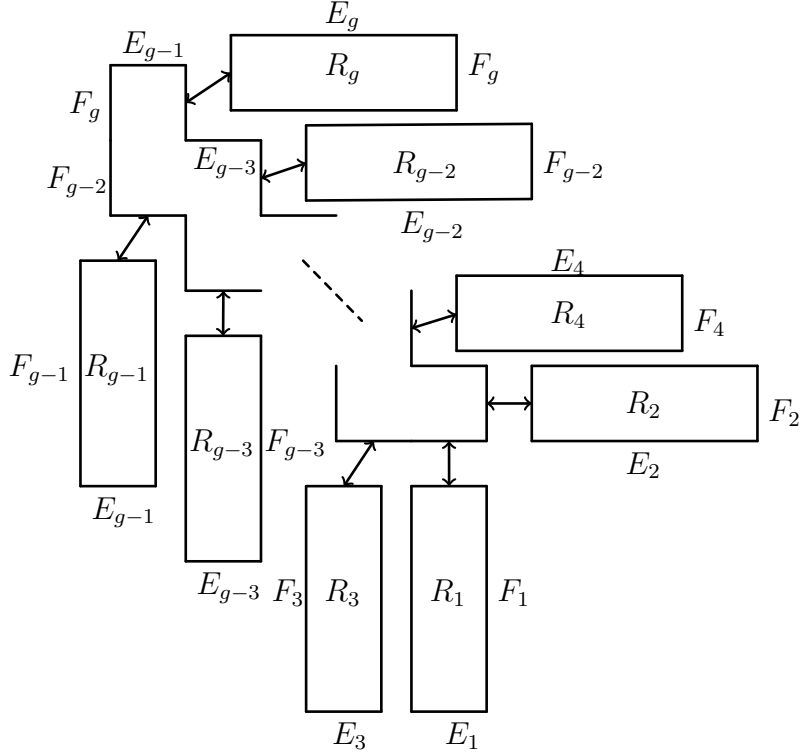
\begin{figure}[h]
\center
\begin{tikzpicture}[line cap=round,line join=round,>=triangle 45,x=1cm,y=1cm]
\clip(-5.5,-4.5) rectangle (5.3,6.3);
\draw [line width=1pt] (-4.4,2.4)-- (-3.4,2.4);
\draw [line width=1pt] (-3.4,2.4)-- (-3.4,-0.6);
\draw [line width=1pt] (-3.4,-0.6)-- (-4.4,-0.6);
\draw [line width=1pt] (-4.4,-0.6)-- (-4.4,2.4);
\draw [line width=1pt] (-4,3)-- (-3,3);
\draw [line width=1pt] (-4,3)-- (-4,4);
\draw [line width=1pt] (-4,4)-- (-4,5);
\draw [line width=1pt] (-4,5)-- (-3,5);
\draw [line width=1pt] (-3,5)-- (-3,4);
\draw [line width=1pt] (-3,4)-- (-2,4);
\draw [line width=1pt] (-1.4,-0.6)-- (-0.4,-0.6);
\draw [line width=1pt] (-0.4,-0.6)-- (-0.4,-3.6);
\draw [line width=1pt] (-0.4,-3.6)-- (-1.4,-3.6);
\draw [line width=1pt] (-1.4,-3.6)-- (-1.4,-0.6);
\draw [line width=1pt] (0,-0.6)-- (1,-0.6);
\draw [line width=1pt] (1,-0.6)-- (1,-3.6);
\draw [line width=1pt] (1,-3.6)-- (0,-3.6);
\draw [line width=1pt] (0,-3.6)-- (0,-0.6);
\draw [line width=1pt] (1,0)-- (0,0);
\draw [line width=1pt] (0,0)-- (-1,0);
\draw [line width=1pt] (-1,0)-- (-1,1);
\draw [line width=1pt] (0,2)-- (0,1);
\draw [line width=1pt] (0,1)-- (1,1);
\draw [line width=1pt] (1,1)-- (1,0);
\draw [line width=1pt] (0.6,2.2)-- (3.6,2.2);
\draw [line width=1pt] (3.6,2.2)-- (3.6,1.2);
\draw [line width=1pt] (3.6,1.2)-- (0.6,1.2);
\draw [line width=1pt] (1.6,1)-- (4.6,1);
\draw [line width=1pt] (4.6,1)-- (4.6,0);
\draw [line width=1pt] (4.6,0)-- (1.6,0);
\draw [line width=1pt] (1.6,0)-- (1.6,1);
\draw [line width=1pt] (0.6,1.2)-- (0.6,2.2);
\draw [line width=1pt] (1.6,3.22)-- (-1.4,3.2);
\draw [line width=1pt] (-1.4,3.2)-- (-1.4,4.2);
\draw [line width=1pt] (-1.4,4.2)-- (1.6,4.22);
\draw [line width=1pt] (1.6,4.22)-- (1.6,3.22);
\draw [line width=1pt] (0.6,4.4)-- (-2.4,4.4);
\draw [line width=1pt] (-2.4,4.4)-- (-2.4,5.4);
\draw [line width=1pt] (-2.4,5.4)-- (0.6,5.4);
\draw [line width=1pt] (0.6,5.4)-- (0.6,4.4);
\draw [line width=1pt] (-3,3)-- (-3,2);
\draw [line width=1pt] (-3,2)-- (-2,2);
\draw [line width=1pt] (-2,4)-- (-2,3);
\draw [line width=1pt] (-3,1.4)-- (-2,1.4);
\draw [line width=1pt] (-2,1.4)-- (-2,-1.6);
\draw [line width=1pt] (-2,-1.6)-- (-3,-1.6);
\draw [line width=1pt] (-3,-1.6)-- (-3,1.4);
\draw [line width=1pt] (-2,3)-- (-1,3);
\draw [line width=1pt,dash pattern=on 3pt off 3pt] (-1.44,2.4)-- (-0.66,1.6);
\draw [line width=1pt, to-to] (0,1.5)-- (0.6,1.6866666666666656);
\draw [line width=1pt, to-to] (1,0.5)-- (1.6,0.5);
\draw [line width=1pt, to-to] (0.5,0)-- (0.5009967405589751,-0.6);
\draw [line width=1pt, to-to] (-0.5,0)-- (-0.8984023818581294,-0.6);
\draw [line width=1pt, to-to] (-2.5035954928659843,1.4)-- (-2.5,2);
\draw [line width=1pt, to-to] (-3.902994615283088,2.4)-- (-3.5,3);
\draw [line width=1pt, to-to] (-3,4.5)-- (-2.4,4.900229320360969);
\draw [line width=1pt, to-to] (-2,3.495235324078027)-- (-1.4,3.698589718276874);
\draw (0.3,-3.6) node[anchor=north west] {$E_1$};
\draw (-1.2,-3.6) node[anchor=north west] {$E_3$};
\draw (-3,-1.6) node[anchor=north west] {$E_{g-3}$};
\draw (-4.4,-0.6) node[anchor=north west] {$E_{g-1}$};
\draw (-4,5.6) node[anchor=north west] {$E_{g-1}$};
\draw (-3.05,4.05) node[anchor=north west] {$E_{g-3}$};
\draw (-4.7,4.8) node[anchor=north west] {$F_{g}$};
\draw (-5,3.8) node[anchor=north west] {$F_{g-2}$};
\draw (4.6,0.7) node[anchor=north west] {$F_2$};
\draw (3.6,1.9) node[anchor=north west] {$F_4$};
\draw (1.6,4) node[anchor=north west] {$F_{g-2}$};
\draw (0.6,5.3) node[anchor=north west] {$F_g$};
\draw (2.7,0) node[anchor=north west] {$E_2$};
\draw (1.7,2.7) node[anchor=north west] {$E_4$};
\draw (-0.3,3.2) node[anchor=north west] {$E_{g-2}$};
\draw (-1.3,6) node[anchor=north west] {$E_g$};
\draw (1,-1.7) node[anchor=north west] {$F_1$};
\draw (-2,-1.7) node[anchor=north west] {$F_3$};
\draw (-2.1,0.3) node[anchor=north west] {$F_{g-3}$};
\draw (-5.5,1.3) node[anchor=north west] {$F_{g-1}$};
\draw (0.1,-1.7) node[anchor=north west] {$R_1$};
\draw (-1.3,-1.7) node[anchor=north west] {$R_3$};
\draw (-3.1,0.3) node[anchor=north west] {$R_{g-3}$};
\draw (-4.5,1.3) node[anchor=north west] {$R_{g-1}$};
\draw (2.7,0.8) node[anchor=north west] {$R_2$};
\draw (1.7,2) node[anchor=north west] {$R_4$};
\draw (-0.4,4) node[anchor=north west] {$R_{g-2}$};
\draw (-1.3,5.3) node[anchor=north west] {$R_g$};
\end{tikzpicture}
\caption{The alternative model for $\Lgn$ is made of a core staircase to which are attached the long rectangles $R_i$. The curve $E_i$ (resp. $F_j$) represents the homology class $e_i$ (resp. $f_j$).}
\label{fig:another_model_2}
\end{figure}

\subsection{An upper bound on $KVol(\Lgn)$}
In this section we provide estimates for KVol on the surface $\Lgn$. Recall from \cite[Section 3]{MM} that the supremum in the definition of KVol can be taken over pairs of simple closed geodesics. In the case of translation surfaces, closed geodesics are homologous to unions of saddle connections. Since saddle connections are closed curves on $\Lgn$ (which has a single singularity), we have:

\[ \KVol(\Lgn) = \sup_{\alpha,\beta \text{ saddle connections}} \frac{\Int(\alpha,\beta)}{l(\alpha)l(\beta)} \]
In this setting, we show:
\begin{Theo}\label{theo:int_alpha_beta}
For any pair of saddle connections $\alpha, \beta$ on $\Lgn$, we have 
\begin{equation}\label{etoile}
 \frac{\Int(\alpha,\beta)}{l(\alpha)l(\beta)} \leq \frac{1}{n}(\frac{n+1}{n})^2+ \frac{6}{n^2}
\end{equation}
\end{Theo}

From this result, we deduce directly that
\[ \KVol(\Lgn) \leq  (g(n+1)-1) (\frac{(n+1)^2}{n^3} + \frac{6}{n^2}). \]
Further, as remarked in Section 2, 
\[ \frac{g(n+1)-1}{n} \leq KVol(\Lgn),\]
so that, in particular, $\KVol(\Lgn) \longrightarrow g$ as $n$ goes to infinity, proving Theorem \ref{theo:main}.

\begin{proof}[Proof of Theorem \ref{theo:int_alpha_beta}]
Let $\alpha$ and $\beta$ be two saddle connections on $\Lgn$. We decompose the homolgy class of $\alpha$ (resp. $\beta$) in the basis $(e_1, \cdots, e_g, f_1, \cdots , f_g)$ of the homology. The first case we deal with is as follows:

\begin{Lem}\label{lem:case_ccef}
For any saddle connection $\alpha$ in $L_ {g,n}$ being in homology an integer combination of the $e_i$, $i$ odd, and the $f_j$, $j$ even, and any saddle connection $\beta$, we have
\[ l(\beta) \geq n |\Int(\alpha,\beta)| \]
In particular
\[ \frac{\Int(\alpha,\beta)}{l(\alpha) l(\beta)} \leq \frac{1}{n}. \]
\end{Lem}
\begin{proof}[Proof of Lemma \ref{lem:case_ccef}]
As seen in the table of the intersections, the non-singular $e_i$ or $f_j$ do not intersect each other, and in particular do not intersect $\alpha$. It follows that if we decompose $\beta$ in the basis of the homology ($e_1, f_1, \cdots, e_g,f_g)$, the intersection $\Int(\alpha,\beta)$ will be at most the number of singular $e_i$ and $f_j$ in the decomposition. But each singular $e_i$ or $f_j$ in the decomposition of $\beta$ corresponds to a trip through a long rectangle $R_i$ and accounts for a length at least $n$, so that:
\[ l(\beta) \geq n |\Int(\alpha,\beta)|. \]
Given that $l(\alpha) \geq 1$, we get 
\[ \frac{\Int(\alpha,\beta)}{l(\alpha) l(\beta)} \leq \frac{1}{n}. \]
\end{proof}

In particular, we deduce from Lemma \ref{lem:case_ccef} that Equation \eqref{etoile} holds if either $\alpha$ or $\beta$ is an integer combination of the non-singular $e_i$ and $f_j$ only. In the rest of the proof, we will assume that neither $\alpha$ nor $\beta$ correspond to such saddle connections. In the alternative model for $\Lgn$, this says exactly that $\alpha$ and $\beta$ have to cross a long rectangle $R_i$.

In particular, we can decompose the saddle connections $\alpha$ and $\beta$ by cutting them each time they enters or leaves a rectangle $R_i$ (lengthwise). This gives a decomposition into smaller (non-closed) segments $\alpha = \alpha_1 \cup \cdots \cup \alpha_k$ (resp. $\beta = \beta_1 \cup \cdots \cup \beta_l$) alternating between:
\begin{enumerate}[label=(\roman*)]
\item long segments (of length at least $n$) inside a long rectangle $R_i$,
\item short segments which stay inside the core staircase of Figure \ref{fig:another_model_2}.
\end{enumerate} 
By convention, we will include the endoints in the short segments, apart from the singularities (the possible singular intersection will be counted separately). Since long segments and short segments are alternating, there are at least $\max(\lfloor k/2 \rfloor, 1)$ long segments and there are at most $\lceil k/2 \rceil$ short segments for $\alpha$. Notice that:
\begin{itemize}
\item Long segments and short segments do not lie in the same part of the surface, hence they cannot intersect. 
\item Any two short segments $\alpha_i$ and $\beta_j$ intersect at most once, as no side of the core staircase is identified to another side of the core staircase.
\item Concerning the intersection of long segments, we have:
\end{itemize}
\begin{Lem}\label{lem:3.2}
Given two long segments $\alpha_i$ and $\beta_j$ in the same rectangle $R$, we have
\[ \frac{\# \alpha_i \cap \beta_j }{l(\alpha_i) l(\beta_j)} \leq \frac{1}{n}(\frac{n+1}{n})^2 + \frac{1}{n^2}\]
where $\# \alpha_i \cap \beta_j$ denotes the cardinal of the set of intersection points.
\end{Lem}
\begin{proof}
The proof of this Lemma is similar to the proof of Proposition 2.5 in \cite{CKMcras}. We first identify the sides of each long rectangle $R$ to form a torus $T$. Then, for each long segment $\alpha_i$ (resp. $\beta_j$) contained in the long rectangle $R$, we construct a closed curve $\tilde{\alpha}_i$ (resp. $\tilde{\beta}_j$) on the corresponding torus $T$. This construction can be done by adding to $\alpha_i$ (resp. $\beta_j$) a small portion of curve of length at most one, and removes at most one intersection, so that
\begin{itemize}
\item[(a)] $\Int(\tilde{\alpha}_i,\tilde{\beta}_j) \geq \# \alpha_i \cap \beta_j -1$.
\item[(b)] $l(\tilde{\alpha}_i) \leq l(\alpha_i) + 1 \text{ and } l(\tilde{\beta}_j) \leq l(\beta_j)+1.$
\end{itemize}
Moreover, $l(\alpha_i) \geq n$ and $l(\beta_j) \geq n$ so that:
\begin{itemize}
\item[(c)] $l(\alpha_i)+1 \leq l(\alpha_i) (1 + 1/n)$ (and the same holds for $\beta_j$), 
\item[(d)] $1 \leq \cfrac{l(\alpha_i)l(\beta_j)}{n^2}$.
\end{itemize}
Now, since $\KVol(T) = 1$ on the flat torus $T$, and given that the rectangle $R$ has area $n$ (and so does the torus $T$), we get:
\[ \frac{\Int(\tilde{\alpha}_i,\tilde{\beta}_j)}{l(\tilde{\alpha}_i) l(\tilde{\beta}_j)} \leq \frac{1}{n}. \]
In particular
\begin{align*}
\#  \alpha_i \cap \beta_j  & \leq \Int(\tilde{\alpha}_i,\tilde{\beta}_j)+1 && \text{ by (a)} \\
& \leq \frac{1}{n}(l(\tilde{\alpha}_i) l(\tilde{\beta}_j))+1 \\
& \leq \frac{1}{n} (l(\alpha_i) + 1)(l(\beta_j) +1)+1 && \text{ by (b)} \\
& \leq \frac{1}{n} l(\alpha_i) (1+\frac{1}{n}) l(\beta_j) (1+\frac{1}{n})+ \frac{l(\alpha_i)l(\beta_j)}{n^2} && \text{ by (c) and (d).}
\end{align*}
This gives Theorem \ref{theo:int_alpha_beta}.
\end{proof}

\textbf{End of the proof.}
Counting all intersections, we have:
\begin{equation*}
\Int(\alpha,\beta) \leq (\sum_{i,j} \# \alpha_i \cap \beta_j ) +1
\end{equation*}
where the added intersection accounts for the possible singular intersection. Using the preceeding estimates, we have:
\begin{align*}
\Int(\alpha,\beta) & \leq \left( \sum_{\alpha_i,\beta_j \text{ long segments}} \# \alpha_i \cap \beta_j \right) + \left( \sum_{\alpha_i,\beta_j \text{ short segments}} \# \alpha_i \cap \beta_j \right) +1\\
& \leq \left( \sum_{\alpha_i,\beta_j \text{ long segments}} (\frac{1}{n}(\frac{n+1}{n})^2 + \frac{1}{n^2}) l(\alpha_i) l(\beta_j) \right) + \left( \sum_{\alpha_i,\beta_j \text{ short segments}}1 \right) +1\\
& \leq (\frac{1}{n}(\frac{n+1}{n})^2 + \frac{1}{n^2}) \left(\sum_{\alpha_i,\beta_j \text{ long segments}} l(\alpha_i) l(\beta_j) \right) + \lceil \frac{k}{2} \rceil \lceil \frac{l}{2} \rceil +1\\
& \leq (\frac{1}{n}(\frac{n+1}{n})^2 + \frac{1}{n^2}) l(\alpha) l(\beta) + \lceil \frac{k}{2} \rceil \lceil \frac{l}{2} \rceil +1
\end{align*}
Now, since there are at least $\max(\lfloor k/2 \rfloor,1)$ long segments of $\alpha$, each long segment having length at least $n$, we get $l(\alpha) \geq \max(\lfloor k/2 \rfloor,1)n$, so that $\frac{k-1}{2} \leq \frac{l(\alpha)}{n}$, and 
\[ \lceil \frac{k}{2} \rceil \leq \frac{k+1}{2} \leq \frac{l(\alpha)}{n} + 1 \leq \frac{2 l(\alpha)}{n} \]
where the last inequality comes from $l(\alpha) \geq n$. Similarly, we have $$ \lceil \frac{l}{2} \rceil \leq \cfrac{l+1}{2} \leq \cfrac{2 l(\beta)}{n}$$
so that 
\begin{align*}
\Int(\alpha,\beta) & \leq (\frac{1}{n} (\frac{n+1}{n})^2 + \frac{1}{n^2}) l(\alpha)l(\beta) + \frac{4}{n^2} l(\alpha) l(\beta)+1 \\
 & \leq (\frac{1}{n} (\frac{n+1}{n})^2 + \frac{5}{n^2})l(\alpha) l(\beta) + \frac{l(\alpha) l(\beta)}{n^2}
\end{align*}
again using that $l(\alpha) \geq n$ and $l(\beta) \geq n$.
\end{proof}

\section{Even spin and hyperelliptic families}\label{sec:hyperelliptic_even}
We conclude this paper by giving a hyperelliptic family of surfaces $\Hgn$ for $g \geq 3$ and an even spin family of surfaces $\Mgn$ (for $g \geq 4$) such that for fixed $g$, $\KVol(\Hgn)$ and $\KVol(\Mgn)$ converge to $g$ as $n$ goes to infinity. 

The proof is in fact similar to the case of $\Lgn$, as each surface can be decomposed into core polygons (giving rise to short segments) and long rectangles (giving rise to long segments). These two families of surfaces have the property that each edge of a core polygon is glued to an edge of a long rectangle, which allows to generalize Lemma \ref{lem:3.2}. Further, the curves staying in the core polygons do not intersect each other and the conclusion of Lemma \ref{lem:case_ccef} can be generalized to these families of surfaces. This allows to give bounds for $\KVol(\Hgn)$ and $\KVol(\Mgn)$ which are easily shown to converge to $g$ as $g$ is fixed and $n$ goes to infinity.

\subsection{The family $\Hgn$}
A convenient way to construct a family of hyperelliptic surfaces is to copy the staircase model of the double regular $(2g+1)$-gon. However, we need each \emph{long rectangle} to have area $n$. One way to do this is to set the lengths of the horizontal and vertical curves $e_i$ and $f_j$ drawn in the left of Figure \ref{fig:Hgn} as 
\[ l(e_i) = n^{\frac{g-i-1}{g-1}} \text{ and } l(f_j) = n^{\frac{j-1}{g-1}}. \]

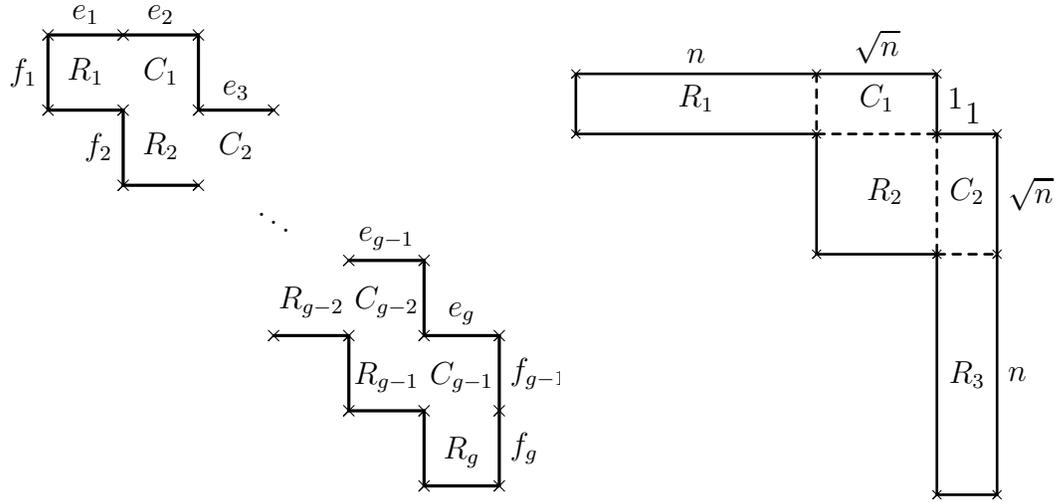
\begin{figure}[h]
\center
\begin{tikzpicture}[line cap=round,line join=round,>=triangle 45,x=1cm,y=1cm]
\clip(-3.5,-1.2) rectangle (3.8,6);
\draw [line width=1.2pt] (-3,5)-- (-2,5);
\draw [line width=1.2pt] (-2,5)-- (-1,5);
\draw [line width=1.2pt] (-1,5)-- (-1,4);
\draw [line width=1.2pt] (-1,4)-- (0,4);
\draw [line width=1.2pt] (-3,5)-- (-3,4);
\draw [line width=1.2pt] (-3,4)-- (-2,4);
\draw [line width=1.2pt] (-2,4)-- (-2,3);
\draw [line width=1.2pt] (-2,3)-- (-1,3);
\draw [line width=1.2pt] (1,2)-- (2,2);
\draw [line width=1.2pt] (2,2)-- (2,1);
\draw [line width=1.2pt] (1,1)-- (1,0);
\draw [line width=1.2pt] (1,0)-- (2,0);
\draw [line width=1.2pt] (2,0)-- (2,-1);
\draw [line width=1.2pt] (2,-1)-- (3,-1);
\draw [line width=1.2pt] (3,-1)-- (3,0);
\draw [line width=1.2pt] (3,0)-- (3,1);
\draw [line width=1.2pt] (3,1)-- (2,1);
\draw [line width=1.2pt] (1,1)-- (0,1);
\draw (-2.5,5) node[above] {$e_1$};
\draw (-1.5,5) node[above] {$e_2$};
\draw (-0.5,4) node[above] {$e_3$};
\draw (1.5,2) node[above] {$e_{g-1}$};
\draw (2.5,1) node[above] {$e_g$};
\draw (-3,4.5) node[left] {$f_1$};
\draw (-2,3.5) node[left] {$f_2$};
\draw (3, 0.5) node[right] {$f_{g-1}$};
\draw (3, -0.5) node[right] {$f_g$};
\draw (-2.5,4.2) node[above] {$R_1$};
\draw (-1.5,4.2) node[above] {$C_1$};
\draw (-1.5,3.2) node[above] {$R_2$};
\draw (-0.5,3.2) node[above] {$C_2$};
\draw (0,2.2) node[above] {$\ddots$};
\draw (0.5,1.1) node[above] {$R_{g-2}$};
\draw (1.5,1.1) node[above] {$C_{g-2}$};
\draw (1.5,0.1) node[above] {$R_{g-1}$};
\draw (2.5,0.1) node[above] {$C_{g-1}$};
\draw (2.5,-0.9) node[above] {$R_g$};
\begin{scriptsize}
\draw [color=black] (-3,5)-- ++(-2pt,-2pt) -- ++(4pt,4pt) ++(-4pt,0) -- ++(4pt,-4pt);
\draw [color=black] (-2,5)-- ++(-2pt,-2pt) -- ++(4pt,4pt) ++(-4pt,0) -- ++(4pt,-4pt);
\draw [color=black] (-1,5)-- ++(-2pt,-2pt) -- ++(4pt,4pt) ++(-4pt,0) -- ++(4pt,-4pt);
\draw [color=black] (-3,4)-- ++(-2pt,-2pt) -- ++(4pt,4pt) ++(-4pt,0) -- ++(4pt,-4pt);
\draw [color=black] (-2,4)-- ++(-2pt,-2pt) -- ++(4pt,4pt) ++(-4pt,0) -- ++(4pt,-4pt);
\draw [color=black] (-2,3)-- ++(-2pt,-2pt) -- ++(4pt,4pt) ++(-4pt,0) -- ++(4pt,-4pt);
\draw [color=black] (-1,4)-- ++(-2pt,-2pt) -- ++(4pt,4pt) ++(-4pt,0) -- ++(4pt,-4pt);
\draw [color=black] (0,4)-- ++(-2pt,-2pt) -- ++(4pt,4pt) ++(-4pt,0) -- ++(4pt,-4pt);
\draw [color=black] (-1,3)-- ++(-2pt,-2pt) -- ++(4pt,4pt) ++(-4pt,0) -- ++(4pt,-4pt);
\draw [color=black] (2,1)-- ++(-2pt,-2pt) -- ++(4pt,4pt) ++(-4pt,0) -- ++(4pt,-4pt);
\draw [color=black] (3,1)-- ++(-2pt,-2pt) -- ++(4pt,4pt) ++(-4pt,0) -- ++(4pt,-4pt);
\draw [color=black] (3,0)-- ++(-2pt,-2pt) -- ++(4pt,4pt) ++(-4pt,0) -- ++(4pt,-4pt);
\draw [color=black] (2,-1)-- ++(-2pt,-2pt) -- ++(4pt,4pt) ++(-4pt,0) -- ++(4pt,-4pt);
\draw [color=black] (2,0)-- ++(-2pt,-2pt) -- ++(4pt,4pt) ++(-4pt,0) -- ++(4pt,-4pt);
\draw [color=black] (1,0)-- ++(-2pt,-2pt) -- ++(4pt,4pt) ++(-4pt,0) -- ++(4pt,-4pt);
\draw [color=black] (3,-1)-- ++(-2pt,-2pt) -- ++(4pt,4pt) ++(-4pt,0) -- ++(4pt,-4pt);
\draw [color=black] (1,1)-- ++(-2pt,-2pt) -- ++(4pt,4pt) ++(-4pt,0) -- ++(4pt,-4pt);
\draw [color=black] (0,1)-- ++(-2pt,-2pt) -- ++(4pt,4pt) ++(-4pt,0) -- ++(4pt,-4pt);
\draw [color=black] (2,2)-- ++(-2pt,-2pt) -- ++(4pt,4pt) ++(-4pt,0) -- ++(4pt,-4pt);
\draw [color=black] (1,2)-- ++(-2pt,-2pt) -- ++(4pt,4pt) ++(-4pt,0) -- ++(4pt,-4pt);
\end{scriptsize}
\end{tikzpicture}
\begin{tikzpicture}[line cap=round,line join=round,>=triangle 45,x=1cm,y=1cm, scale=0.8]
\clip(-4.1,-4.1) rectangle (4,4);
\draw [line width=1pt] (-4,3)-- (0,3);
\draw [line width=1pt] (0,3)-- (2,3);
\draw [line width=1pt] (2,3)-- (2,2);
\draw [line width=1pt] (2,2)-- (3,2);
\draw [line width=1pt] (3,2)-- (3,0);
\draw [line width=1pt] (3,0)-- (3,-4);
\draw [line width=1pt] (3,-4)-- (2,-4);
\draw [line width=1pt] (2,-4)-- (2,0);
\draw [line width=1pt] (2,0)-- (0,0);
\draw [line width=1pt] (0,0)-- (0,2);
\draw [line width=1pt] (0,2)-- (-4,2);
\draw [line width=1pt] (-4,2)-- (-4,3);
\draw [line width=1pt,dash pattern=on 3pt off 3pt] (0,2)-- (2,2);
\draw [line width=1pt,dash pattern=on 3pt off 3pt] (2,2)-- (2,0);
\draw [line width=1pt,dash pattern=on 3pt off 3pt] (0,3)-- (0,2);
\draw [line width=1pt,dash pattern=on 3pt off 3pt] (2,0)-- (3,0);
\draw (-2,2.2) node[above] {$R_1$};
\draw (0.65,0.65) node[anchor=south west] {$R_2$};
\draw (2.5,0.65) node[above] {$C_2$};
\draw (2.5,-1.6) node[below] {$R_3$};
\draw (1,2.2) node[above] {$C_1$};
\draw (-2,3) node[above] {$n$};
\draw (1,3) node[above] {$\sqrt{n}$};
\draw (2,2.6) node[right] {$1$};
\draw (2.6,2) node[above] {$1$};
\draw (3,1) node[right] {$\sqrt{n}$};
\draw (3,-2) node[right] {$n$};
\begin{scriptsize}
\draw [color=black] (0,0)-- ++(-2pt,-2pt) -- ++(4pt,4pt) ++(-4pt,0) -- ++(4pt,-4pt);
\draw [color=black] (0,2)-- ++(-2pt,-2pt) -- ++(4pt,4pt) ++(-4pt,0) -- ++(4pt,-4pt);
\draw [color=black] (2,0)-- ++(-2pt,-2pt) -- ++(4pt,4pt) ++(-4pt,0) -- ++(4pt,-4pt);
\draw [color=black] (-4,2)-- ++(-2pt,-2pt) -- ++(4pt,4pt) ++(-4pt,0) -- ++(4pt,-4pt);
\draw [color=black] (-4,3)-- ++(-2pt,-2pt) -- ++(4pt,4pt) ++(-4pt,0) -- ++(4pt,-4pt);
\draw [color=black] (0,3)-- ++(-2pt,-2pt) -- ++(4pt,4pt) ++(-4pt,0) -- ++(4pt,-4pt);
\draw [color=black] (2,3)-- ++(-2pt,-2pt) -- ++(4pt,4pt) ++(-4pt,0) -- ++(4pt,-4pt);
\draw [color=black] (2,2)-- ++(-2pt,-2pt) -- ++(4pt,4pt) ++(-4pt,0) -- ++(4pt,-4pt);
\draw [color=black] (3,2)-- ++(-2pt,-2pt) -- ++(4pt,4pt) ++(-4pt,0) -- ++(4pt,-4pt);
\draw [color=black] (3,0)-- ++(-2pt,-2pt) -- ++(4pt,4pt) ++(-4pt,0) -- ++(4pt,-4pt);
\draw [color=black] (2,-4)-- ++(-2pt,-2pt) -- ++(4pt,4pt) ++(-4pt,0) -- ++(4pt,-4pt);
\draw [color=black] (3,-4)-- ++(-2pt,-2pt) -- ++(4pt,4pt) ++(-4pt,0) -- ++(4pt,-4pt);
\end{scriptsize}
\end{tikzpicture}
\caption{On the left, a combinatorial model for $\Hgn$. On the right, the example of $H_3^n$.}
\label{fig:Hgn}
\end{figure}

Next, we distinguish the core polygons $C_i$ and the big rectangles $R_i$ and proceed with the proof as in the case of $\Lgn$. Notice that the $e_i$'s (resp. the $f_j$'s) are pairwise non-intersecting and that the intersection of the $e_i$ and the $f_j$ is given by the following table:

\begin{center}
\begin{tabular}{| c || c c c c c c |}
\hline
 $\Int(e_i,f_j)$ & $e_1$ & $e_2$ & $e_3$ & $e_4$ & $e_5$ & $\cdots$ \\ 
\hline\hline
 $f_1$ & 1 & 0 & 0  & 0 & 0 & $\cdots$ \\  
 $f_2$ & -1 & 1  & 0  & 0 & 0 & $\cdots$ \\  
 $f_3$ & 1  & -1 & 1 & 0 & 0 & $\cdots$ \\  
 $f_4$ & -1  & 1  & -1  & 1 & 0 & $\cdots$ \\  
 $f_5$ & 1 & -1  & 1  & -1 & 1 & $\cdots$ \\  
 $\cdots$ & $\cdots$ & $\cdots$  & $\cdots$  & $\cdots$   & $\cdots$  & $\cdots$ \\   
\hline
\end{tabular}
\end{center}

This allows to get an adapted version of Lemma \ref{lem:case_ccef}:

\begin{Lem}
The closed saddle connections $\gamma$ contained in the core polygons correspond to the homology classes $e_{i}$, $f_{i-1}$ and $e_i + f_{i-1}$ for $2 \leq i \leq g$.\newline
For any such saddle connection $\gamma$ and any other saddle connection $g$, we have:
\[ \frac{\Int(\gamma,g)}{l(\gamma)l(g)} \leq \frac{1}{n}. \]
\end{Lem}

Further, similarly to Lemma \ref{lem:3.2}, we have:

\begin{Lem}
For any two saddle connections $\alpha$ and $\beta$ which are not contained in the core polygons $C_i$, we have
\[ \frac{\Int(\alpha,\beta)}{l(\alpha)l(\beta)} \leq \frac{1}{n}(1+n^{-\frac{1}{g}})^2 + \frac{6}{n^2}. \]
\end{Lem}

Using that the area of $\Hgn$ is $gn + (g-1)n^{\frac{g-1}{g}}$, we get that
\[ g + (g-1)n^{\frac{-1}{g}} \leq KVol(\Hgn) \leq (g + (g-1)n^{\frac{-1}{g}})((1+n^{-\frac{1}{g}})^2 + \frac{6}{n}), \]
where the lover bound comes from the fact that $e_1$ and $f_1$ are intersecting once, and $l(e_1) l(f_1) = n$. Hence, for fixed $g$, $\KVol(\Hgn)$ goes to $g$ as $n$ goes to infinity.

\subsection{The family $\Mgn$}
Similarly to $\Lgn$ and $\Hgn$, it is possible to construct an even spin family of translation surfaces $\Mgn$, such that for any fixed $g \geq 4$, $\KVol(\Mgn)$ goes to $g$ as $n$ goes to infinity. For example, construct each $\Mgn$ from $H_3^n$ by adding steps as in $\Lgn$, see Figure \ref{fig:Mgn}. As we have seen in the case of $\Lgn$, the operation of adding steps do not change the parity of the spin structure. In particular, the parity of the spin structure of $\Mgn$ is the same as $H_3^n$, that is even. Further, similarly to $\Lgn$, the surface $\Mgn$ is not hyperelliptic as an hyperelliptic involution would have to fix each cylinder, and hence act as an involution of $C_2$, $R_3$ and $C_3$ but also $C_2 \cup R_3 \cup C_3$, which is impossible. In the case of $\Mgn$, an argument similar to the previous cases show:

\[ \frac{gn + 2\sqrt{n} + (g-3)}{n} \leq \KVol(\Mgn) \leq \frac{gn + 2\sqrt{n} + (g-3)}{n}((1 + \frac{1}{\sqrt{n}})^2 + \frac{6}{n}). \]

\begin{figure} 
\center
\begin{tikzpicture}[line cap=round,line join=round,>=triangle 45,x=1cm,y=1cm, scale=0.65]
\clip(-10.5,-4.7) rectangle (4,8.7);
\draw [line width=1pt] (0,0)-- (2,0);
\draw [line width=1pt] (0,0)-- (0,2);
\draw [line width=1pt] (0,2)-- (-4,2);
\draw [line width=0.8pt,dash pattern=on 3pt off 3pt] (-4,2)-- (-4,3);
\draw [line width=1pt] (-4,3)-- (0,3);
\draw [line width=1pt] (0,3)-- (2,3);
\draw [line width=1pt] (2,3)-- (2,2);
\draw [line width=1pt] (2,2)-- (3,2);
\draw [line width=1pt] (3,2)-- (3,0);
\draw [line width=1pt] (2,0)-- (2,-4);
\draw [line width=1pt] (2,-4)-- (3,-4);
\draw [line width=1pt] (3,-4)-- (3,0);
\draw [line width=1pt] (-4,2)-- (-5,2);
\draw [line width=1pt] (-5,2)-- (-5,3);
\draw [line width=0.8pt,dash pattern=on 3pt off 3pt] (-5,3)-- (-4,3);
\draw [line width=1pt] (-4,3)-- (-4,7);
\draw [line width=0.8pt,dash pattern=on 3pt off 3pt] (-4,7)-- (-5,7);
\draw [line width=1pt] (-5,7)-- (-5,3);
\draw [line width=1pt] (-4,7)-- (-4,8);
\draw [line width=1pt] (-4,8)-- (-5,8);
\draw [line width=0.8pt,dash pattern=on 3pt off 3pt] (-5,8)-- (-5,7);
\draw [line width=1pt] (-5,7)-- (-9,7);
\draw [line width=1pt] (-9,7)-- (-9,8);
\draw [line width=1pt] (-5,8)-- (-9,8);
\draw [line width=0.8pt,dash pattern=on 3pt off 3pt] (0,3)-- (0,2);
\draw [line width=0.8pt,dash pattern=on 3pt off 3pt] (0,2)-- (2,2);
\draw [line width=0.8pt,dash pattern=on 3pt off 3pt] (2,2)-- (2,0);
\draw [line width=0.8pt,dash pattern=on 3pt off 3pt] (2,0)-- (3,0);
\draw (-9,7.5) node[left] {$1$};
\draw (-7,8) node[above] {$n$};
\draw (-5,5) node[left] {$n$};
\draw (-2,3) node[above] {$n$};
\draw (1,3) node[above] {$\sqrt{n}$};
\draw (3,1) node[right] {$\sqrt{n}$};
\draw (2.5,-4) node[below] {$1$};
\draw (3,-2) node[right] {$n$};
\draw (-4.5,8) node[above] {$1$};
\draw (-5,2.5) node[left] {$1$};
\draw (-7,7) node[above] {$R_1$};
\draw (-4.5,7) node[above] {$C_1$};
\draw (-5,5) node[right] {$R_2$};
\draw (-4.5,2) node[above] {$C_2$};
\draw (-2,2) node[above] {$R_3$};
\draw (1,2) node[above] {$C_3$};
\draw (1,0.5) node[above] {$R_4$};
\draw (2.5,0.5) node[above] {$C_4$};
\draw (3,-2) node[left] {$R_5$};
\begin{scriptsize}
\draw [color=black] (0,0)-- ++(-1.5pt,-1.5pt) -- ++(3pt,3pt) ++(-3pt,0) -- ++(3pt,-3pt);
\draw [color=black] (2,0)-- ++(-1.5pt,-1.5pt) -- ++(3pt,3pt) ++(-3pt,0) -- ++(3pt,-3pt);
\draw [color=black] (0,2)-- ++(-1.5pt,-1.5pt) -- ++(3pt,3pt) ++(-3pt,0) -- ++(3pt,-3pt);
\draw [color=black] (-4,2)-- ++(-1.5pt,-1.5pt) -- ++(3pt,3pt) ++(-3pt,0) -- ++(3pt,-3pt);
\draw [color=black] (-4,3)-- ++(-1.5pt,-1.5pt) -- ++(3pt,3pt) ++(-3pt,0) -- ++(3pt,-3pt);
\draw [color=black] (0,3)-- ++(-1.5pt,-1.5pt) -- ++(3pt,3pt) ++(-3pt,0) -- ++(3pt,-3pt);
\draw [color=black] (2,3)-- ++(-1.5pt,-1.5pt) -- ++(3pt,3pt) ++(-3pt,0) -- ++(3pt,-3pt);
\draw [color=black] (2,2)-- ++(-1.5pt,-1.5pt) -- ++(3pt,3pt) ++(-3pt,0) -- ++(3pt,-3pt);
\draw [color=black] (3,2)-- ++(-1.5pt,-1.5pt) -- ++(3pt,3pt) ++(-3pt,0) -- ++(3pt,-3pt);
\draw [color=black] (3,0)-- ++(-1.5pt,-1.5pt) -- ++(3pt,3pt) ++(-3pt,0) -- ++(3pt,-3pt);
\draw [color=black] (2,-4)-- ++(-1.5pt,-1.5pt) -- ++(3pt,3pt) ++(-3pt,0) -- ++(3pt,-3pt);
\draw [color=black] (3,-4)-- ++(-1.5pt,-1.5pt) -- ++(3pt,3pt) ++(-3pt,0) -- ++(3pt,-3pt);
\draw [color=black] (-5,2)-- ++(-1.5pt,-1.5pt) -- ++(3pt,3pt) ++(-3pt,0) -- ++(3pt,-3pt);
\draw [color=black] (-5,3)-- ++(-1.5pt,-1.5pt) -- ++(3pt,3pt) ++(-3pt,0) -- ++(3pt,-3pt);
\draw [color=black] (-4,7)-- ++(-1.5pt,-1.5pt) -- ++(3pt,3pt) ++(-3pt,0) -- ++(3pt,-3pt);
\draw [color=black] (-5,7)-- ++(-1.5pt,-1.5pt) -- ++(3pt,3pt) ++(-3pt,0) -- ++(3pt,-3pt);
\draw [color=black] (-4,8)-- ++(-1.5pt,-1.5pt) -- ++(3pt,3pt) ++(-3pt,0) -- ++(3pt,-3pt);
\draw [color=black] (-5,8)-- ++(-1.5pt,-1.5pt) -- ++(3pt,3pt) ++(-3pt,0) -- ++(3pt,-3pt);
\draw [color=black] (-9,7)-- ++(-1.5pt,-1.5pt) -- ++(3pt,3pt) ++(-3pt,0) -- ++(3pt,-3pt);
\draw [color=black] (-9,8)-- ++(-1.5pt,-1.5pt) -- ++(3pt,3pt) ++(-3pt,0) -- ++(3pt,-3pt);
\end{scriptsize}
\end{tikzpicture}
\caption{The surface $M_5^n$.}
\label{fig:Mgn}
\end{figure}
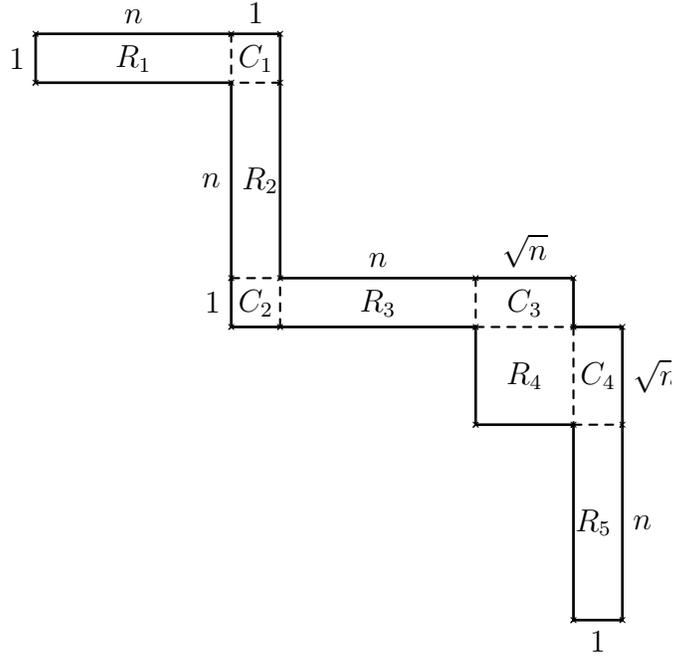

\end{document}